\documentclass{amsart}
\usepackage{amssymb, amsfonts, amsthm}
\usepackage{graphics, epstopdf, pstricks, psfrag}
\usepackage[margin =1 in]{geometry}

\title[Variations on a theme of Kasteleyn]{Variations on a theme of Kasteleyn, with application to the totally nonnegative Grassmannian}

\author{David E.\ Speyer}

\address{Department of Mathematics, University of Michigan, 
         Ann Arbor, MI 48109}

\theoremstyle{plain}
\newtheorem{theorem}{Theorem}
\newtheorem{conjecture}{Conjecture}[section]

\newtheorem{Proposition}[conjecture]{Proposition}

\newtheorem{corollary}[conjecture]{Corollary}

\theoremstyle{definition}

\theoremstyle{remark}
\newtheorem{remark}[conjecture]{Remark}
\newtheorem{Remark}[conjecture]{Remark}

\newtheorem*{remark*}{Remark}

\newcommand{\cross}{\mathrm{cross}}

\newcommand{\RR}{\mathbb{R}}

\newcommand{\DD}{\mathbb{D}}

\newcommand{\newword}[1]{\textbf{\emph{#1}}}

\newcommand{\Pf}{\mathrm{Pf}}

\begin{document}

\maketitle

\begin{abstract} We provide a short proof of a classical result of Kasteleyn, and prove several variants thereof. One of these results has become key in the parametrization of positroid varieties, and thus deserves the short direct proof which we provide.
\end{abstract}


Let $G$ be a finite graph. A \newword{perfect matching} of $G$ is a collection $M$ of edges of $G$ such that every vertex of $G$ is contained in precisely one edge of $M$. Perfect matchings are also called \newword{dimer configurations}. Kasteleyn~\cite{Kast} established the following:

\begin{theorem} \label{Thrm1}
Let $G$ be a planar bipartite graph with $N$ black vertices $b_1$, $b_2$, \dots, $b_N$ and $N$ white vertices $w_1$, $w_2$, \dots, $w_N$. There is a $N \times N$ matrix $K$ with $K_{ij} = \pm 1$ if there is an edge from $i$ to $j$, and $K_{ij}=0$ otherwise, such that $\det K$ is the number of perfect matchings of $G$.
\end{theorem}

We will define a \newword{graph with boundary} to be a finite graph $G$ with a specified subset $\partial G$ of the vertices of $G$, equipped with a circular ordering. We call the vertices of $\partial G$ the \newword{boundary vertices} and the other vertices of $G$ the \newword{internal vertices}. We define a \newword{perfect matching of a graph with boundary} to be a collection $M$ of edges of $G$ which contains each internal vertex precisely once, and each boundary vertex at most once. 
For a perfect matching $M$, we define $\partial M$ to be the set of boundary vertices contained in $M$.
For a subset $I$ of $\partial G$, we define $\DD(G, I)$ to be the number of perfect matchings $M$ of $I$ with $\partial M = I$.
We define a \newword{graph with boundary to be embedded} in a disc if $G$ is embedded in a closed planar disc $D$ such that the the vertices of $\partial G$ lie on $\partial D$, in their specified order.

The first variant of Kasteleyn's result which we want is:
\begin{theorem} \label{Thrm2}
Let $G$ be a bipartite graph with boundary embedded in a disk, such that all of the boundary vertices are white. Let there be $N+k$ black vertices $b_1$, $b_2$, \dots, $b_{N+k}$, let there be $N$ internal white vertices $w_1$, $w_2$, \dots, $w_N$ and let there be $n$ boundary vertices $w_{N+1}$, \dots, $w_{N+n}$, with that circular order. Then there is an $(N+k) \times (N+n)$ matix $K$ with $K_{ij} = \pm 1$ if there is an edge from $i$ to $j$, and $K_{ij}=0$ otherwise, having the following property: For any $k$ element subset $I$ of $\partial G$, let $K_I$ be the submatrix of $K$ using all rows, the first $N$ columns, and the additional $k$ columns indexed by $I$. Then $\DD(G, I) = \det K_I$. 
\end{theorem}

This will then imply
\begin{theorem} \label{Thrm3}
Let $G$ be as in Theorem~\ref{Thrm2}. Then there is a $k \times n$ real matrix $L$ with the property we now describe: For any $k$ element subset $I$ of $\partial G$, let $L_I$ be the submatrix of $L$ using all rows and the columns indexed by $I$. Then $\DD(G, I) = \det L_I$. 
\end{theorem}
In particular, all maximal minors of $L_I$ are nonnegative. 
As we will point out explicitly in Corollary~\ref{In Grassmannian}, that means that the $\binom{n}{k}$ numbers $\DD(G,I)$ are the Pl\"ucker coordinates of a point in Postnikov's \newword{totally nonnegative Grassmannian}.

Kasteleyn also proved a version of his theorem for graphs which are not bipartite:
\begin{theorem} \label{Thrm4}
Let $G$ be a planar graph with $N$ vertices $v_1$, $v_2$, \dots, $v_{N}$. Then there is an $N \times N$ skew symmetric matrix $X$ with $X_{ij} = - X_{ji} =  \pm 1$ if there is an edge from $i$ to $j$, and $X_{ij}=0$ otherwise, such that the Pfaffian $\Pf(X)$ is the number of perfect matchings of $G$.
\end{theorem}

We will prove this and the corresponding results:
\begin{theorem}\label{Thrm5}
Let $G$ be a planar graph with boundary embedded in a disc, having $N$ internal vertices $v_1$, $v_2$, \dots, $v_N$ and $n$ boundary vertices $v_{N+1}$, $v_{N+2}$, \dots, $v_{N+n}$ in that circular order. Then there is an $(N+n) \times (N+n)$ skew symmetric matrix $X$ with $X_{ij} = - X_{ji} =  \pm 1$ if there is an edge from $i$ to $j$, and $X_{ij}=0$ otherwise having the following property: For any subset $I$ of $\partial G$, let $X_I$ be the submatrix of $X$ using the first $N$ rows and first $N$ columns, and additionally those rows and columns indexed by $I$. Then $\DD(G, I) = \Pf(X_I)$.
\end{theorem}

\begin{theorem}\label{Thrm6}
Let $G$ be as in Theorem~\ref{Thrm5} and assume $\DD(G, \emptyset)>0$ (which implies that $N$ is even). Then there is an $n \times n$ skew symmetric real matrix $Y$ with the following property: For any subset $I$ of $\partial G$, let $Y_I$ be the submatrix using the rows and columns indexed by $I$. Then $\DD(G, I) = \Pf(Y_I) \  \DD(G, \emptyset)$ for all subsets $I$ of $\partial G$. 
\end{theorem}

\begin{remark*}
We take the Pfaffian of an odd by odd skew symmetric matrix to be zero, so the theorems are true but trivial in the cases that involve such Pfaffians.
\end{remark*}

\begin{remark*}
Theorem~\ref{Thrm6} shows that there are many skew-symmetric matrices all of whose Pfaffian's are nonnegative; it would be interesting to develop analogues of classical results on nonnegative matrices for Pfaffians.
\end{remark*}

\begin{remark*}
Theorems~3.1 and~4.1 of~\cite{Kuo2} are Theorems~\ref{Thrm6} and~\ref{Thrm3} with the added hypotheses (respectively) that $\DD(G, \emptyset)=1$ and that $\DD(G,I_0)=1$ for some $I_0$.
\end{remark*}

All of these theorems have easy variants where there are weights on the edges of $G$. Specifically, let $w: \mathrm{Edges}(G) \to \RR_{>0}$ be any weighting function. For a perfect matching $M$, we define $w(M) = \prod_{e \in M} w(e)$; we define $\DD(G,I,w) = \sum_{\partial M = I} w(M)$. Then the corresponding results are
\begin{enumerate}
\item In the setting of Theorem~\ref{Thrm1}, there is a matrix $K$ with $K_{ij} = \pm w((b_i,w_j))$ if $(b_i,w_j)$ is an edge, and $0$ otherwise, such that $\det K = \sum_M w(M)$.
\item In the setting of Theorem~\ref{Thrm2}, there is a matrix $K$ with $K_{ij} = \pm w((b_i,w_j))$ or $0$ as above such that $\det K_I = \DD(G,I,w)$.
\item In the setting of Theorem~\ref{Thrm3}, there is a real matrix $L$ such that $\det L_I = \DD(G,I,w)$.
\item In the setting of Theorem~\ref{Thrm4}, there is a skew-symmetric matrix $X$ with $X_{ij}=-X_{ji} = \pm w((v_i,v_j))$ if $(v_i,v_j)$ is an edge, and $0$ otherwise, such that $\det X = \sum_M w(M)$.
\item In the setting of Theorem~\ref{Thrm5},  there is a skew-symmetric matrix $X$ with $X_{ij}=-X_{ji} = \pm w((v_i,v_j))$ or $0$ as above such that $\Pf(X_I) = \DD(G,I,w)$.
\item In the setting of Theorem~\ref{Thrm6} (including the hypothesis that $ \DD(G, \emptyset,w)>0$), there is a real skew-symmetric matrix $Y$ such that $\DD(G,I,w)= \Pf(Y_I)  \DD(G, \emptyset,w)$.
\end{enumerate}

In particular, part~(3) shows that the $\binom{n}{k}$ numbers $\DD(G,I,w)$, as $I$ varies, are the Pl\"ucker coordinates of a point in the totally nonnegative Grassmannian.
If we fix $G$ and let $w$ vary over all possible weightings of the edges of $G$, we thus obtain a parametrization of a portion of the totally nonnegative Grassmannian.
This parametrization of the totally nonnegative Grassmannian is the one found by Postnikov~\cite{Post}, who described it in terms of certain random walks. 
Talaska~\cite{Talaska} recast Postnikov's formulas in terms of flows. Postnikov, Williams and the author~\cite{PSW} implicitly pointed out that this was equivalent to summing over matchings.
Lam's lecture notes~\cite{Lam} make the point explicit. As the positive Grassmannian and its parametrizations grow more popular, the author feels that there should be a brief paper which records a direct proof that $w \mapsto (\DD(G,I,w))_{I \in \binom{[n]}{k}}$ parametrizes a portion of the totally nonnegative Grassmannian.

The author would like to express his gratitude to Jim Propp for introducing him to Kasteleyn's method, Kuo's condensation theorem, and  the pleasures of combinatorial research.

\section{A topological proof of Theorems~\ref{Thrm1}, \ref{Thrm2}, \ref{Thrm4} and~\ref{Thrm5}}

The key to our proof is to prove a more general result for non-planar graphs. Let $G$ be a general graph. We will define a \newword{planar immersion} of $G$ to be a continuous map $\phi: G \to \RR^2$ such that each edge of $G$ is taken to a line segment and, for any edge $e$, and any vertex $v$ not an end point of $e$, the point $\phi(v)$ is not contained in $\phi(e)$. We point out explicitly that a line segment has positive length; a single point is not a line segment.

For $G$ a graph, $\phi : G \to \RR^2$ a planar immersion and $M$ a perfect matching of $G$, we define $\cross(M)$ to be the number of unordered pairs $\{ e_1, e_2 \}$ of distinct edges of $M$ for which $\phi(e_1)$ and $\phi(e_2)$ intersect. We set $\epsilon(M) = (-1)^{\cross(M)}$. 

\begin{Proposition} \label{Thrm1Topology}
Let $G$ be a bipartite graph with $N$ black vertices $b_i$ and $N$ white vertices $w_j$, equipped with a planar immersion $\phi$. Then there is an $N \times N$ matrix $K_{ij}$, with $K_{ij} = \pm 1$ if there is an edge from $b_i$ to $w_j$, and $0$ otherwise, such that $\det K = \sum_M \epsilon(M)$.
\end{Proposition}

\begin{proof}
We first note that the theorem is easy if all the black vertices $\phi(b_i)$ lie in order on a line and all the white vertices $\phi(w_j)$ lie in order on a parallel line; just take $K_{ij}=1$ whenever there is an edge $(b_i, w_j)$. We must check that $\epsilon(M)$ is the sign coming from the determinant. Let $(b_1, w_{\sigma(1)})$, $(b_2, w_{\sigma(2)})$, \dots, $(b_N, w_{\sigma(N)})$ be a perfect matching $M$, so $\sigma$ is a permutation. The number of crossing edges of $M$ is the number of $(i_1, i_2)$ such that $i_1 < i_2$ and $\sigma(i_1) > \sigma(i_2)$. So $\cross(M)$ is the number of inversions of $\sigma$, and $\epsilon(M)$ is the sign of $\sigma$, as desired.

Given any point $z \in (\RR^2)^{\mathrm{Vertices}(G)}$, there is a map $\phi_z : G \to \RR^2$ which sends a vertex $v$ to $z(v)$ and sends each edge to a line segment or point.
Let $\Omega$ be the set of $z$ for which $\phi_z$ is a planar immersion. 

We note that $\Omega$ is an open subset of $ (\RR^2)^{\mathrm{Vertices}(G)} \cong \RR^{4N}$, obtained by deleting the codimension $1$ subvarieties on which certain triples of vertices become colinear. Let $\Omega' \supset \Omega$ be the open set where we impose that all the vertices have distinct images $z(v)$, and also that the interiors of $\phi_z(e_1)$ and $\phi_z(e_2)$  do not overlap in a line segment for any distinct edges $e_1$ and $e_2$.
So $\Omega'$ is obtained from $\RR^{4N}$ by deleting the codimension two subvarieties on which pairs of vertices become equal, or on which certain quadruples of vertices become colinear. 

In particular, since $\Omega'$ is $\RR^{4N}$ with codimension two subvarieties removed, $\Omega'$ is connected. Let $z_0$ be a point of $\Omega$ corresponding to the immersion in the first paragraph, so Proposition~\ref{Thrm1Topology} is true at $\phi_{z_0}$. Let $z_1$ be any point of $\Omega$, and choose a path $z(t)$ from $z_0$ to $z_1$ through $\Omega'$. 
We will show that the Proposition is true for every $\phi_{z(t)}$ with $z(t) \in \Omega$.

As we travel along $z(t)$, the only changes of the topology of the embedding occur when a vertex $v$ passes through an edge $e$. Let $\epsilon_1$ and $\epsilon_2$ denote the sign functions for the two topologies. We claim that $\epsilon_1(M) = - \epsilon_2(M)$ if $e \in M$, and $\epsilon_1(M) = \epsilon_2(M)$ otherwise. This is because $M$ has exactly one edge incident to $v$. This edge crosses $e$ in one topology and not the other, and no other crossings change. So, if $K$ is a matrix which works for the first topology, then we can obtain a matrix for the second topology by switching the sign of the entry corresponding to edge $e$. Since we have a matrix which works at $z_0$, we obtain a matrix which works at any $z$.
\end{proof}

\begin{proof}[Proof of Theorem~\ref{Thrm1}]
Fary~\cite{Fary} showed that any planar graph can be drawn so that the edges are straight lines. Drawing $G$ in this manner, the function $\epsilon$ is simply $1$ and we obtain Theorem~\ref{Thrm1}.
\end{proof}

\begin{Remark}
We could avoid appealing to Fary's theorem by using piecewise linear edges and adding additional coordinates for the positions of the bends in these edges. This creates a few new cases, but no significant difficulties.
\end{Remark}

\begin{Remark}
Norine~\cite{Norine} used planar immersions to provide a characterization of Pfaffian graphs, and thus implicitly to provide a proof of Kasteleyn's theorem, but did not  discuss deforming the immersion.
The author previously posted a sketch of this argument on Mathoverflow~\cite{MO199465}. 
The author is not aware of any other prior sources for this argument.
\end{Remark}

We have proved Theorem~\ref{Thrm1}. 
Slight variants of this argument prove Theorems~\ref{Thrm2}, \ref{Thrm4} and~\ref{Thrm5}. For Theorem~\ref{Thrm2}, we fix a closed rectangle $D$ and consider immersions $G \to D$ taking $\partial D \to \partial G$ in the specified circular order. Our starting point is that the black vertices occurs in order on the top edge of the rectangle and the black vertices occur in order on the bottom edge. The fact that $D$ is convex ensures that $\phi(G) \subset D$ if $\phi(\mathrm{Vertices}(G)) \subset D$.

For Theorems~\ref{Thrm4} and~\ref{Thrm5}, we proceed similarly. Our starting point is now to take all the vertices $\phi(v)$ to lie on the boundary of a circle, and recall that one way to describe the signs occurring in the Pfaffian is as the number of crossings when a matching is drawn in this manner.

\section{Proofs of Theorems~\ref{Thrm3} and~\ref{Thrm6}}

We begin with Theorem~\ref{Thrm3}. Let $K$ be the matrix from Theorem~\ref{Thrm2}. If $G$ has no perfect matchings, the Theorem is immediate; take $L=0$.
So we may assume that $G$ has perfect matchings, and thus that $\det K_I \neq 0$ for some $I$. In particular, the first $N$ columns of $K$ are linearly independent.

Therefore, applying row operations, we may transform $K$ into a matrix of block form
\[ \begin{pmatrix}
\mathrm{Id}_N & \ast \\ 0 & L 
\end{pmatrix} \]
without changing any maximal minors. Then, in the notations of Theorems~\ref{Thrm2} and~\ref{Thrm3}, we have $\det K_I = \det L_I$.
This proves Theorem~\ref{Thrm3}.

The proof of Theorem~\ref{Thrm6} is similar. If $G$ has no perfect matchings, take $Y=0$. Otherwise, the upper left $N \times N$ submatrix, $X_{\emptyset}$ must be of rank $N$. 
So we can write $S X_{\emptyset} S^{-1}= J_N$ where $J_N$ is the $N \times N$ matrix
\[ J_N = \begin{pmatrix}
0 & 1 & & & & & \\
-1 & 0 & & & & & \\
& & 0 & 1 & & &  \\
& & -1 & 0 &  & & \\
& & & & \ddots & & \\
& & & & & 0 & 1 \\
& & & & & -1 & 0 \\
\end{pmatrix}. \]
Left and right multiplying $X$ by $\left(\begin{smallmatrix} S & 0 \\ 0 & \mathrm{Id}_n \end{smallmatrix} \right)$ and its transpose, we obtain a matrix of the form $\left( \begin{smallmatrix} J_N & E \\ -E^T & Z \end{smallmatrix} \right)$. Further symmetric row operations can bring us to the form $\left( \begin{smallmatrix} J_N & 0 \\ 0 & Y\end{smallmatrix} \right)$. 
For any subset $I$ of $\partial G$, we have $\Pf(X_I) = \det(S) \Pf(Y_I)$. Since $\Pf(Y_{\emptyset})=1$, we deduce that $\Pf(X_I) = \Pf(Y_I) \Pf(X_{\emptyset}) = \Pf(Y_I) \DD(G, \emptyset)$ as desired.

\section{Relation to results of Kuo and to the totally nonnegative Grassmannian}

In the case where $\#(\partial G)=4$, Theorems~\ref{Thrm3} and~\ref{Thrm6} imply two results of Eric Kuo~\cite{Kuo1, Kuo2}.

\begin{corollary}[Kuo's condensation relation for bipartite graphs] \label{KuoRelate1}
Let $G$ be a bipartite planar graph with circularly ordered white boundary $\{ a,b,c,d \}$, and with two more white vertices than black vertices. Then
\[ \DD(\{ a,c \} ) \DD(\{ b,d \}) = \DD(\{ a,b \}) \DD(\{ c,d \}) + \DD(\{ a,d \}) \DD(\{ b,c \}) . \]
\end{corollary}

\begin{proof}
For any $2 \times 4$ matrix $\left( \begin{smallmatrix} x_1 & x_2 & x_3 & x_4 \\ y_1 & y_2 & y_3 & y_4 \end{smallmatrix} \right)$, we have the Pl\"ucker relation
\[ \det \begin{pmatrix} x_1 & x_3 \\ y_1 & y_3 \end{pmatrix} \det \begin{pmatrix} x_2 & x_4 \\ y_2 & y_4 \end{pmatrix}  = 
\det \begin{pmatrix} x_1 & x_2 \\ y_1 & y_2 \end{pmatrix} \det \begin{pmatrix} x_3 & x_4 \\ y_3 & y_4 \end{pmatrix} + 
\det \begin{pmatrix} x_1 & x_4 \\ y_1 & y_4 \end{pmatrix} \det \begin{pmatrix} x_2 & x_3 \\ y_2 & y_3 \end{pmatrix} . \qedhere \]
\end{proof}
%

\begin{corollary}[Kuo's condensation relation for general graphs] \label{KuoRelate2}
Let $G$ be a planar graph with circularly ordered boundary $\{ a,b,c,d \}$, and an even number of vertices. Then
\[ \DD(\{ a,c \} ) \DD(\{ b,d \}) + \DD(\emptyset) \DD(\{ a,b,c,d \}) = \DD(\{ a,b \}) \DD(\{ c,d \}) + \DD(\{ a,d \}) \DD(\{ b,c \}) . \]
\end{corollary}

\begin{proof}
We are being asked to prove, for a $4 \times 4$ skew symmetric matrix $Y$, that $Y_{13} Y_{24} + \Pf(Y) = Y_{12} Y_{34} + Y_{14} Y_{23}$, which is immediate.
\end{proof}

\begin{remark}
Many authors have pointed out that Corollaries~\ref{KuoRelate1} and~\ref{KuoRelate2} and similar identities follow from properties of Pfaffians. (See~\cite{YZ}, \cite{Fulmek}, \cite{Knuth}.) It seems uncommon, though, to observe that the relations that occur when deleting subsets of $n$ boundary vertices are precisely the relations between the Pfaffians of an $n \times n$ matrix.
\end{remark}

There is no need to limit ourselves to $2 \times 4$ matrices. More generally, we deduce the following result:
\begin{corollary} \label{In Grassmannian}
Let $G$ be as in Theorems~\ref{Thrm2} and~\ref{Thrm3}. Then the $\binom{n}{k}$ numbers $\DD(I)$, as $I$ ranges through $k$-element subsets of $\partial G$, are either all zero or the Pl\"ucker coordinates of a point on the Grassmannian $G(k,n)$.
\end{corollary}

\begin{proof}
By definition, the Pl\"ucker coordinates of $\mathrm{RowSpan}(L)$ are the maximal minors of $L$, assuming $L$ has rank $k$.
\end{proof}

As explained in the introduction, this last result is key to parametrizations of the totally nonnegative Grassmannian.




\newpage
\raggedright

\end{document}